\documentclass[12pt]{amsart}
\usepackage{amsmath,amssymb,amsthm,amsrefs}
\usepackage{enumitem}
\usepackage{color,hyperref}
\usepackage{color}
\hypersetup{colorlinks,breaklinks,
	linkcolor=blue,urlcolor=blue,
	anchorcolor=blue,citecolor=blue}
\usepackage{cleveref}
\theoremstyle{plain}%note that this is the default style (to learn)
\newtheorem{theorem}{Theorem}[section]
\newtheorem{proposition}[theorem]{Proposition}
\newtheorem{corollary}[theorem]{Corollary}
\newtheorem{lemma}[theorem]{Lemma}
\theoremstyle{definition}

\newtheorem{remark}[theorem]{Remark}

\makeatletter
\@addtoreset{equation}{section}
\makeatother

\makeatletter
\newcommand{\Spvek}[2][r]{%
  \gdef\@VORNE{1}
  \left(\hskip-\arraycolsep%
    \begin{array}{#1}\vekSp@lten{#2}\end{array}%
  \hskip-\arraycolsep\right)}

\def\vekSp@lten#1{\xvekSp@lten#1;vekL@stLine;}
\def\vekL@stLine{vekL@stLine}
\def\xvekSp@lten#1;{\def\temp{#1}%
  \ifx\temp\vekL@stLine
  \else
    \ifnum\@VORNE=1\gdef\@VORNE{0}
    \else\@arraycr\fi%
    #1%
    \expandafter\xvekSp@lten
  \fi}
\makeatother

\begin{document}
\title[Uniqueness theorems for meromorphic inner functions]
{Uniqueness theorems for meromorphic inner functions and canonical systems\\}
\author{Burak Hat\.{i}no\u{g}lu}
\address{Department of Mathematics, Michigan State University, East Lansing MI 48824, U.S.A.}
\email{hatinogl@msu.edu}

\subjclass[2020]{30, 34}

%\date{August 17, 2019}


\begin{abstract}
We prove uniqueness problems for meromorphic inner functions on the upper half-plane. In these problems we consider spectral data depending partially or fully on the spectrum, derivative values at the spectrum, Clark measure or the spectrum of the negative of a meromorphic inner function. Moreover we consider applications of these uniqueness results to inverse spectral theory of canonical Hamiltonian systems and obtain generalizations of Borg-Levinson two-spectra theorem for canonical Hamiltonian systems and unique determination of a Hamiltonian from its spectral measure under some conditions.
\end{abstract}
\maketitle

\section{Introduction}\label{Sec1}
Inner functions on $\mathbb{C}_+$ are bounded analytic functions with unit modulus almost everywhere on the boundary $\mathbb{R}$. If an inner function extends to $\mathbb{C}$ meromorphically, it is called a meromorphic inner function (MIF), which is usually denoted by $\Theta$. 

MIFs appear in various fields such as spectral theory of differential operators \cite{BBP17,MP05}, Krein-de Branges theory of entire functions \cite{B68,MP010,R17}, functional model theory \cite{N86,N02}, approximation theory \cite{P015}, Toeplitz operators \cite{HM16,MP10,P18} and non-linear Fourier transform \cite{P21}. MIFs also play a critical role in the Toeplitz approach to the uncertainty principle \cite{MP05,P15}, which was used in the study of problems of harmonic analysis \cite{MP15,P12,P13}.

Each MIF is represented by the Blaschke/singular factorization
$$\Theta(z) = Ce^{iaz}\prod_{n \in \mathbb{N}} c_n \frac{z-\omega_n}{z-\overline{\omega}_n},$$ 
where $C$ is a unimodular constant, $a$ is a nonzero real constant, $\{\omega_n\}_{n \in \mathbb{N}}$ is a discrete sequence (i.e. has no finite accumulation point) in $\mathbb{C}_+$ satisfying the Blaschke condition $\sum_{n \in \mathbb{N}}\mathrm{Im}\omega_n/(1+|\omega_n|^2) < \infty$, and $c_n = (i+\overline{\omega}_n)/(i+\omega_n)$. MIFs are also represented as $\Theta(x) =  \exp(i\phi(x))$ on the real line, where $\phi$ is an increasing real analytic function. 

A complex-valued function is said to be real if it maps real numbers to real numbers on its domain. A meromorphic Herglotz function (MHF) $m$ is a real meromorphic function with positive imaginary part on $\mathbb{C}_+$. It has negative imaginary part on $\mathbb{C}_-$ through the relation $m(\overline{z}) = \overline{m}(z)$. There is a one-to-one correspondence between MIFs and MHFs given by the equations 

$$\Theta = \frac{m_{\Theta}-i}{m_{\Theta}+i} \qquad \text{and} \qquad m_{\Theta} = i\frac{1+\Theta}{1-\Theta}.$$

Therefore MIFs can be described by parameters $(b,c,\mu)$ via the Herglotz representation
 $m_{\Theta}(z) = bz + c + iS\mu_{\Theta}(z)$,
where $b \geq 0$, $c \in \mathbb{R}$ and 
$$
S\mu_{\Theta}(z) := \frac{1}{i\pi}\int\Big(\frac{1}{t-z}-\frac{t}{1+t^2}\Big)d\mu_{\Theta}(t)
$$ 
is the Schwarz integral of the positive discrete Poisson-finite measure $\mu_{\Theta}$ on $\mathbb{R}$, i.e. $\int 1/(1+t^2) d\mu_{\Theta}(t) < \infty$.
The number $\pi b$ is considered as the point mass of $\mu_{\Theta}$ at infinity. This extended measure $\mu_{\Theta}$ is called the Clark (or
spectral) measure of $\Theta$. The spectrum of $\Theta$, denoted by $\sigma(\Theta)$, is the level set $\{\Theta = 1\}$, so $\mu_{\Theta}$ is supported on $\sigma(\Theta)$ or $\sigma(\Theta) \cup \{\infty\}$. The point masses at $t \in \sigma({\Theta})$ are given by $\mu_{\Theta}(t) = 2\pi/|\Theta'(t)|$ and the point mass at infinity is non-zero if and only if $\lim_{y \rightarrow +\infty}\Theta(iy) = 1$ and the limit $\lim_{y \rightarrow +\infty}y^2\Theta'(iy)$ exists. All of these above mentioned properties of MIFs can be found e.g in \cite{P15} and references therein.

Existence, uniqueness, interpolation and other complex function theoretic problems for MIFs were studied in various papers \cite{BCLRS16,P15,PR16,R17}. In this paper we consider unique determination of the MIF $\Theta$ from several spectral data and conditions depending fully or partially on $\sigma(\Theta)$, $\sigma(-\Theta)$ and $\{\mu_{\Theta}(t)\}_{t \in \sigma(\Theta)}$. The two spectra $\sigma(\Theta)$ and $\sigma(-\Theta)$ are the sets of poles and zeros of $m_{\Theta}$ respectively. They are interlacing, discrete sequences on $\mathbb{R}$, so we denote them by $\{a_n\}_{n \in \mathbb{Z}} = \sigma(\Theta) := \{\Theta = 1\}$ and $\{b_n\}_{n \in \mathbb{Z}} = \sigma(-\Theta) := \{\Theta = -1\}$ indexed in increasing order. We also introduce the disjoint intervals $I_n := (a_n,b_n)$. If $\sigma(\Theta)$ is bounded below, then $\{a_n\}$, $\{b_n\}$ and $\{I_n\}$ are indexed by $\mathbb{N}$.

MIFs (equivalently MHFs) appear in inverse spectral theory. Inverse spectral problems aim to determine an operator from given spectral data. Classical results were given for Sturm-Liouville operators and can be found e.g. in \cite{HAT} and references therein. 

Canonical Hamiltonian systems are the most general class of symmetric second-order operators, so the more classical second-order operators such as Schr\"{o}dinger, Jacobi, Dirac, Sturm–Liouville operators and Krein strings can be transformed to canonical Hamiltonian systems \cite{REM18}, e.g. see Section 1.3 of \cite{REM18} for rewriting the Schr\"{o}dinger equation as a canonical Hamiltonian system.

A canonical Hamiltonian system is a $2 \times 2$ differential equation system of the form 
\begin{equation}\label{cHs}
  J f'(x) = -zH(x)f(x), 
\end{equation} 
where $L > 0$, $x \in (0,L)$ and
$$
J:= {\begin{pmatrix}
			0 & -1 \\
			1 & 0 
		\end{pmatrix}}, \quad f(x) = {\begin{pmatrix}
			f_1(x) \\
			f_2(x) 
		\end{pmatrix}}.
$$
We assume $H(x) \in \mathbb{R}^{2\times 2}$,  $H \in L_{\text{loc}}^1(0,L)$ and $H(x) \geq 0$ a.e. on $(0,L)$. In this paper we consider the limit circle case, whixh means $\int_0^L \text{Tr} H(x) dx < \infty$, on a finite interval, i.e. $0 < L < \infty$. The self-adjoint realizations of a canonical Hamiltonian system in the limit circle case with separated boundary conditions are described by
    \begin{equation}\label{BCs}
       f_1(0)\sin\alpha-f_2(0)\cos\alpha = 0, \qquad f_1(L)\sin\beta-f_2(L)\cos\beta = 0,
    \end{equation}
 with $\alpha, \beta \in [0,\pi)$. Such a self-adjoint system has a discrete spectrum $\sigma_{\alpha,\beta}$.
 
 Let $f = f (x, z)$ be a solution of $Jf' = -zHf$ with the boundary condition $f_1(L)\sin\beta-f_2(L)\cos\beta = 0$. If we let $u$ and $v$ be solutions of $Ju'= -zHu$ with the initial conditions $u_1(0) = \cos(\alpha)$, $u_2(0) = \sin(\alpha)$ and $v_1(0) = -\sin(\alpha)$, $v_2(0) = \cos(\alpha)$, then the solution $u$ satisfies the boundary condition at $x = 0$, and $f(x,z) = v(x,z) + m_{\alpha,\beta}(z)u(x,z)$ is the unique solution of the form $v + Mu$ that satisfies the boundary condition $\beta$ at $x = L$. The complex function $m_{\alpha,\beta}$ is called the Weyl $m$-function, which is 
 \begin{equation}\label{id2}
  m_{\alpha,\beta}(z) := \frac{\cos(\alpha)f_1(0) + \sin(\alpha)f_2(0)}{-\sin(\alpha)f_1(0) + \cos(\alpha)f_2(0)}.
 \end{equation} The spectrum $\sigma_{\alpha,\beta} = \{a_n\}_n$ is the set of poles of $m_{\alpha,\beta}$. The norming constant $\gamma_{\alpha,\beta}^{(n)}$ for $a_n \in \sigma_{\alpha,\beta}$ is defined as $\gamma_{\alpha,\beta}^{(n)} := 1/||u(\cdot,a_n)||^2$. The Weyl $m$-function $m_{\alpha,\beta}(z)$ is a MHF, so the corresponding MIF is defined as $\Theta_{\alpha,\beta} = (m_{\alpha,\beta}-i)/(m_{\alpha,\beta}+i)$, which is called the Weyl inner function.
The Clark (spectral) measure of $\Theta_{\alpha,\beta}$ is the discrete measure supported on $\sigma_{\alpha,\beta}$ with point masses given by the corresponding norming constants, i.e. $\mu_{\alpha,\beta} = \sum_{n} \gamma_{\alpha,\beta}^{(n)}\delta_{a_n}$ is the spectral measure. The MIF $\Theta_{\alpha,\beta}$ carries out spectral properties of the corresponding canonical Hamiltonian system through \eqref{id2}, so we will use results on MIFs in inverse spectral theory of canonical Hamiltonian system.

The paper is organized as follows.
 
Section \ref{SecRes} includes following uniqueness results for MIFs and canonical systems.

\begin{itemize}
    \item In Theorem \ref{uniqres1} we consider unique determination of $\Theta$ from its spectrum, point masses (or equivalently derivative values) at the spectrum and one of the constants $\Theta(0)$, $\lim_{y \rightarrow +\infty}\Theta(iy)$ or $\prod_{n \in \mathbb{Z}}a_n/b_n \neq 1$  with the condition that $\{|I_n|/(1+\text{dist}(0,I_n))\}_{n \in \mathbb{Z}} \in l^1(\mathbb{Z})$.
    \item In Theorem \ref{uniqres2}, this condition is replaced by the condition $(1+|x|)^{-1} \in L^1(\mu_{\Theta})$ and that $\mu_{\Theta}$ has no point mass at infinity, so the Clark measure $\mu_{\Theta}$ and one of the constants $\Theta(0)$, $\lim_{y \rightarrow +\infty}\Theta(iy)$ or $\prod_{n \in \mathbb{Z}}a_n/b_n \neq 1$ uniquely determine $\Theta$ in this case.
    \item Theorem \ref{uniqres3} shows that if $\sigma(\Theta)$ is bounded below and $a_n < b_n$, then $\Theta$ is uniquely determined by the spectral data of Theorem \ref{uniqres1} without any other condition. Its second part deals with the case $a_n > b_n$.
    \item In Theorems \ref{uniqres4} and \ref{uniqres5}, we prove that the knowledge of the derivative values from Theorems \ref{uniqres1} and \ref{uniqres3} respectively, can be partially replaced by the knowledge of the corresponding points from the second spectrum $\sigma(-\Theta)$.
     \item Theorem \ref{invspecprb1} shows unique determination of a Hamiltonian from the spectral measure $\mu_{\alpha_1,\beta}$ and boundary conditions $\alpha_1, \alpha_2$ and $\beta$ with the condition that $\{|I_n|/(1+\text{dist}(0,I_n))\}_{n \in \mathbb{Z}} \in l^1(\mathbb{Z})$, where endpoints of $I_n$ are given by the two spectra $\sigma_{\alpha_1,\beta}$ and $\sigma_{\alpha_2,\beta}$.
    \item In Theorem \ref{invspecprb2} we consider unique determination of a Hamiltonian from a spectral measure and boundary conditions in the case that the corresponding spectrum is bounded below.
    \item In Theorem \ref{invspecprb3} and Theorem \ref{invspecprb4} we prove that some missing norming constants from the spectral data of Theorems \ref{invspecprb1} and \ref{invspecprb2} respectively, can be compensated by the corresponding data from a second spectrum.
    \item As Corollary \ref{invspecprb5} we obtain a generalization of Borg-Levison's classical two-spectra theorem on Schr\"{o}dinger operators to canonical Hamiltonian systems.
\end{itemize}

Section \ref{Sec2} includes proofs of results for MIFs.

Section \ref{Sec3} includes proofs of results for canonical Hamiltonian systems.

\section{Results}\label{SecRes}
We first obtain uniqueness theorems for meromporphic inner functions (MIFs) and then consider inverse spectral problems for canonical Hamiltonian systems as applications. However, let us start by fixing our notations. We denote MIFs by $\Theta$ (or $\Phi$) and the corresponding MHFs by $m_{\Theta}$ (or $m_{\Phi}$). Since the spectrum $\sigma(\Theta)$ of a MIF $\Theta$ is a discrete sequence in $\mathbb{R}$, we denote it by $\{a_n\}_{n \in \mathbb{Z}}$. Similarly the second spectrum $\sigma(-\Theta)$ is denoted by $\{b_n\}_{n \in \mathbb{Z}}$. Since $\Theta =  \exp(i\phi)$ on $\mathbb{R}$ for an increasing real analytic function $\phi$ and $\{a_n\}_{n \in \mathbb{Z}}$ and $\{b_n\}_{n \in \mathbb{Z}}$ are interlacing sequences, we get $\{x\in\mathbb{R}~|~\mathrm{Im}\Theta(x) > 0\} = \{(a_n,b_n)\}$. We denote the intervals $(a_n,b_n)$ by $I_n$. A sequence of disjoint intervals $\{J_n\}$ on $\mathbb{R}$ is called short if 
\begin{equation}\label{shortseq}
\sum_{n}\frac{|J_n|^2}{1+\text{dist}^2(0,J_n)} < \infty,   
\end{equation}
and long otherwise. Here $|J_n|$ and $\text{dist}(0,J_n)$ denote the length of the interval $J_n$ and its distance to the origin respectively. Collections of short (and long) intervals appear in various areas of harmonic analysis (see \cite{P15} and references therein). If we assume $0 \notin \cup I_n$, then condition \eqref{shortseq} is nothing but $\{|J_n|/\text{dist}(0,J_n)\} \in l^2$. In some of our results we use the condition that this sequence for $\{I_n\}$ belongs to $l^1$, i.e.
\begin{equation}\label{short1seq}
  \sum_{n}\frac{|I_n|}{1+\text{dist}(0,I_n)} < \infty.  
\end{equation}

In our first theorem, with condition \eqref{short1seq}, we consider unique determination of a MIF from its spectrum and derivative values on the spectrum.

\begin{theorem}\label{uniqres1}
Let $\Theta$ be a MIF, $\sigma(\Theta) = \{a_n\}_{n \in \mathbb{Z}}:= \{\Theta = 1\}$, $\sigma(-\Theta) = \{b_n\}_{n \in \mathbb{Z}}:= \{\Theta = -1\}$ and $I_n := (a_n,b_n)$. If $$\sum_{n \in \mathbb{Z}}\frac{|I_n|}{1+\mathrm{dist}(0,I_n)} < \infty,$$
then the spectral data consisting of 
\begin{itemize}
    \item $\{a_n\}_{n \in \mathbb{Z}}$ 
    \item $\{\Theta'(a_n)\}_{n \in \mathbb{Z}}$ (or equivalently $\{\mu_{\Theta}(a_n)\}_{n \in \mathbb{Z}}$) and 
    \item $L:=\displaystyle\lim_{y \rightarrow +\infty}\Theta(iy)$ or $C:= \Theta(0)$ or $p:=\prod_{n \in \mathbb{Z}}a_n/b_n \neq 1$ 
\end{itemize}
uniquely determine $\Theta$.
\end{theorem}

If the Clark measure has no point mass at infinity, summability condition \eqref{short1seq} can be replaced by an integrability condition on the Clark measure.

\begin{theorem}\label{uniqres2}
Let $\Theta$ be a MIF, $\mu_{\Theta}$ be its Clark measure, $\sigma(\Theta) = \{a_n\}_{n \in \mathbb{Z}}:= \{\Theta = 1\}$ and $\sigma(-\Theta) = \{b_n\}_{n \in \mathbb{Z}}:= \{\Theta = -1\}$. If  $$ \int \frac{1}{1+|x|} d\mu_{\Theta}(x) < \infty$$
and $\mu_{\Theta}$ has no point mass at $\infty$, then $\mu_{\Theta}$ and $L:=\lim_{y \rightarrow +\infty}\Theta(iy)$ or $C:= \Theta(0)$ or $p:=\prod_{n \in \mathbb{Z}}a_n/b_n \neq 1$ uniquely determine $\Theta$.
\end{theorem}

If the spectrum of a MIF is bounded below, then summability condition of Theorem \ref{uniqres1} is not required.

\begin{theorem}\label{uniqres3}
Let $\Theta$ be a MIF, $\sigma(\Theta) = \{a_n\}_{n \in \mathbb{N}}:= \{\Theta = 1\}$ and $\sigma(-\Theta) = \{b_n\}_{n \in \mathbb{N}}:= \{\Theta = -1\}$. 
\begin{enumerate}
    \item If $\sigma(\Theta)=\{a_n\}_{n \in \mathbb{N}}$ is bounded below and $a_n < b_n$, then the spectral data consisting of 
\begin{itemize}
    \item $\{a_n\}_{n \in \mathbb{N}}$ 
    \item $\{\Theta'(a_n)\}_{n \in \mathbb{N}}$ %(or equivalently $\{\mu_{\Theta}(a_n)\}_{n \in \mathbb{N}}$) 
    and 
    \item $C:= \Theta(0)$
\end{itemize}
uniquely determine $\Theta$.
    \item If $\sigma(\Theta)=\{a_n\}_{n \in \mathbb{N}}$ is bounded below and $a_n > b_n$, then the spectral data consisting of 
\begin{itemize}
    \item $\{b_n\}_{n \in \mathbb{N}}$ 
    \item $\{\Theta'(b_n)\}_{n \in \mathbb{N}}$ and 
    \item $C:= \Theta(0)$
\end{itemize}
uniquely determine $\Theta$.
\end{enumerate}
\end{theorem}

Next two theorems show that missing part of point masses of the Clark measure from the data of Theorems \ref{uniqres1} and \ref{uniqres3} can be compensated by the corresponding data from the second spectrum $\sigma(-\Theta)$.

\begin{theorem}\label{uniqres4}
 Let $\Theta$ be a MIF, $\sigma(\Theta) = \{a_n\}_{n \in \mathbb{Z}}:= \{\Theta = 1\}$, $\sigma(-\Theta) = \{b_n\}_{n \in \mathbb{Z}}:= \{\Theta = -1\}$ and $A \subseteq \mathbb{Z}$. If  
\begin{equation*}
\sum_{n \in \mathbb{Z}}\frac{|I_n|}{1+\mathrm{dist}(0,I_n)} < \infty,
\end{equation*}
 then the spectral data consisting of 
\begin{itemize}
    \item $\{a_n\}_{n \in \mathbb{Z}}$ 
    \item $\{b_n\}_{n \in \mathbb{Z}\setminus A}$
    \item $\{\Theta^{'}(a_n)\}_{n \in A}$ (or $\{\mu(a_n)\}_{n \in A}$) and
    \item $L:=\displaystyle\lim_{y \rightarrow +\infty}\Theta(iy)$ or $C:= \Theta(0)$ or $p:=\prod_{n \in A}a_n/b_n \neq 1$
\end{itemize}
uniquely determine $\Theta$.
\end{theorem}

\begin{theorem}\label{uniqres5}
Let $\Theta$ be a MIF, $\sigma(\Theta) = \{a_n\}_{n \in \mathbb{N}}:= \{\Theta = 1\}$, $\sigma(-\Theta) = \{b_n\}_{n \in \mathbb{N}}:= \{\Theta = -1\}$ and $A \subseteq \mathbb{N}$. 
\begin{enumerate}
    \item If $\{a_n\}_{n \in \mathbb{N}}$ is bounded below and $a_n < b_n$, then the spectral data consisting of 
\begin{itemize}
    \item $\{a_n\}_{n \in \mathbb{N}}$ 
    \item $\{b_n\}_{n \in \mathbb{N}\setminus A}$
    \item $\{\Theta^{'}(a_n)\}_{n \in A}$ (or $\{\mu(a_n)\}_{n \in A}$) and
    \item $C:= \Theta(0)$
\end{itemize}
uniquely determine $\Theta$.
    \item If $\{a_n\}_{n \in \mathbb{N}}$ is bounded below and $a_n > b_n$, then the spectral data consisting of 
\begin{itemize}
    \item $\{a_n\}_{n \in \mathbb{N}\setminus A}$ 
    \item $\{b_n\}_{n \in \mathbb{N}}$
    \item $\{\Theta^{'}(b_n)\}_{n \in A}$ and
    \item $C:= \Theta(0)$
\end{itemize}
uniquely determine $\Theta$.
\end{enumerate}
\end{theorem}

Next, we consider inverse spectral theorems on canonical systems as applications of the above mentioned results. We follow the definitions and notations we introduced in Section \ref{Sec1}.

\begin{theorem}\label{invspecprb1}
Let $L > 0$, $\alpha_1,\alpha_2,\beta \in [0,\pi)$, $\alpha_1 \neq \alpha_2$ and $H$ be a trace normed canonical system on $[0,L]$. Also let $\sigma_{\alpha_1,\beta} = \{a_n\}_{n \in \mathbb{Z}}$, $\sigma_{\alpha_2,\beta} = \{b_n\}_{n \in \mathbb{Z}}$ and $I_n = (a_n,b_n)$. If $$\sum_{n \in \mathbb{Z}}\frac{|I_n|}{1+\mathrm{dist}(0,I_n)} < \infty,$$
then the spectral measure $\mu_{\alpha_1,\beta}$ and boundary conditions $\alpha_1$, $\alpha_2$ and $\beta$ uniquely determine $H$.
\end{theorem}

\begin{theorem}\label{invspecprb2}
Let $L > 0$, $\alpha,\beta \in [0,\pi)$  and $H$ be a trace normed canonical Hamiltonian system on $[0,L]$. If $\sigma_{\alpha,\beta} = \{a_n\}_{n \in \mathbb{N}}$ is bounded below and $a_n < b_n$, then the spectral measure $\mu_{\alpha,\beta}$ and boundary conditions $\alpha$ and $\beta$ uniquely determine $H$.
\end{theorem}

The next two results show that the missing norming constants from the spectral data can be compensated by the corresponding data from a second spectrum.

\begin{theorem}\label{invspecprb3}
 Let $L > 0$, $\alpha_1,\alpha_2,\beta \in [0,\pi)$, $\alpha_1 \neq \alpha_2$, $A \subseteq \mathbb{Z}$ and $H$ be a trace normed canonical Hamiltonian system on $[0,L]$. Also let $\sigma_{\alpha_1,\beta} = \{a_n\}_{n \in \mathbb{Z}}$, $\sigma_{\alpha_2,\beta} = \{b_n\}_{n \in \mathbb{Z}}$ and $I_n = (a_n,b_n)$. If  
\begin{equation*}
\sum_{n \in \mathbb{Z}}\frac{|I_n|}{1+\mathrm{dist}(0,I_n)} < \infty,
\end{equation*}
 then the spectral data consisting of 
\begin{itemize}
    \item $\{a_n\}_{n \in \mathbb{Z}}$
    \item $\{b_n\}_{n \in \mathbb{Z}\setminus A}$
    \item $\{\gamma_{\alpha_1,\beta}^{(n)}\}_{n \in A}$ (or $\{\mu_{\alpha_1,\beta}(a_n)\}_{n \in A}$)
    \item $\alpha_1$, $\alpha_2$ and $\beta$
\end{itemize}
uniquely determine $H$.
\end{theorem}

\begin{theorem}\label{invspecprb4}
Let $L > 0$, $\alpha_1,\alpha_2,\beta \in [0,\pi)$, $\alpha_1 \neq \alpha_2$, $A \subseteq \mathbb{N}$ and $H$ be a trace normed canonical Hamiltonian system on $[0,L]$. Also let $\sigma_{\alpha_1,\beta} = \{a_n\}_{n \in \mathbb{N}}$ and $\sigma_{\alpha_2,\beta} = \{b_n\}_{n \in \mathbb{N}}$. If $\sigma_{\alpha_1,\beta}$ is bounded below and $a_n < b_n$, then the spectral data consisting of 
\begin{itemize}
    \item $\{a_n\}_{n \in \mathbb{N}}$
    \item $\{b_n\}_{n \in \mathbb{N}\setminus A}$
    \item $\{\gamma_{\alpha_1,\beta}^{(n)}\}_{n \in A}$ (or $\{\mu_{\alpha_1,\beta}(a_n)\}_{n \in A}$)
    \item $\alpha_1$, $\alpha_2$ and $\beta$
\end{itemize}
uniquely determine $H$.
\end{theorem}

By letting $A = \emptyset$, we get a canonical Hamiltonian system version of Borg-Levinson's classical two spectra theorem for Schr\"{o}dinger operators. Note that the spectrum of a Schr\"{o}dinger (Sturm-Liouville) operator on a finite interval is always bounded below.
\begin{corollary}\label{invspecprb5}
Let $L > 0$, $\alpha_1,\alpha_2,\beta \in [0,\pi)$ and $H$ be a trace normed canonical Hamiltonian system on $[0,L]$. If $\sigma_{\alpha_1,\beta} = \{a_n\}_{n \in \mathbb{N}}$ is bounded below and $a_n < b_n$, then the two spectra $\sigma_{\alpha_1,\beta} = \{a_n\}_{n \in \mathbb{N}}$, $\sigma_{\alpha_2,\beta} = \{b_n\}_{n \in \mathbb{N}}$ and $\alpha_1$, $\alpha_2$ and $\beta$ uniquely determine $H$.
\end{corollary}

\begin{remark}
Schr\"{o}dinger (Sturm-Liouville) operators on finite intervals are characterized by two spectra with specific asymptotics (see e.g. (2.4)-(2.7) in \cite{HAT}). Namely, two interlacing, discrete, bounded below subsets of the real-line satisfying these asymptotics correspond to two spectra of a Schr\"{o}dinger operator on a finite interval and any pair of spectra of a Schr\"{o}dinger operator on a finite interval satisfy these properties. 

Corollary \ref{invspecprb5} shows that unique recovery from two spectra is not related with the special asymptotics of eigenvalues from Schr\"{o}dinger (Sturm-Liouville) class if the given order relation between the two spectra is satisfied. It extends the two-spectra theorem (or Borg-Levinson's theorem) to a broader class of canonical Hamiltonian systems on finite intervals with bounded below spectrum.  
\end{remark}

\begin{remark}
    In this paper we consider canonical Hamiltonian systems on finite intervals to guarantee existence of a discrete spectrum. In general, one can let $L=\infty$ in the definition of canonical Hamiltonian systems \eqref{cHs}. With the restriction of having a discrete spectrum, the inverse spectral theory results above may be obtained for canonical Hamiltonian systems on $\mathbb{R}_{+}$, since our results were obtained from the uniqueness results for MIF, which were considered in the general complex function theoretic framework. As a special case, similar versions of mixed spectral data results were obtained for semi-infinite Jacobi operators with discrete spectrum in \cite{HAT2}.
\end{remark}

\section{Proofs for meromorphic inner functions}\label{Sec2}

To prove our uniqueness theorems, we use an infinite product representation result.
\begin{lemma}\label{lmm1}\emph{\textbf{(}\cite{LEV64}, \textit{Theorem VII.1}\textbf{)}}
  The MHF $m_{\Theta}$ corresponding to the MIF $\Theta$ has the infinite product representation
  \begin{equation}\label{mfcn1}
   m_{\Theta}(z) = i\frac{1+\Theta(z)}{1-\Theta(z)} = c \prod_{n\in \mathbb{Z}}\left(\frac{z}{b_{n}}-1\right)\left(\frac{z}{a_n}-1\right)^{-1}
  \end{equation}
  where $c > 0$, $a_n < b_n < a_{n+1}$ and the product converges normally on $\displaystyle \mathbb{C} \text{\textbackslash} \cup_{n \in \mathbb{N}} a_n$.
 \end{lemma}
%\begin{proof}
%The interlacing property of $\{a_n\}$ and $\{b_n\}$ implies convergence of the infinite products in \eqref{mfcn1} and \eqref{mfcn2}. Then Lemma 4.1 in \cite{HAT} proves these representations for the m-function of the Schr\"{o}dinger operator on $(0,1)$ with Dirichlet-Dirichlet boundary conditions. The m-function is a meromorphic Herglotz function, its sets of poles and zeros are bounded below and satisfy some asymptotic properties. However the arguments used in the proof of that lemma do not depend on these properties of sets and poles, so those arguments can be generalized to any meromorphic Herglotz function. 
%\end{proof}
Note that if $m$ is a MHF, then $-1/m$ is also a MHF. Therefore the roles of zeros and poles can be swapped in Lemma \ref{lmm1} by letting the coefficient $c = m_{\Theta}(0)$ be negative. 

\begin{remark}
If $\sigma(\Theta)$ is bounded below and $a_n < b_n$, then representation \eqref{mfcn1} becomes
\begin{equation}\label{mfcn1b}
   m_{\Theta}(z) = c \prod_{n\in \mathbb{N}}\left(\frac{z}{b_{n}}-1\right)\left(\frac{z}{a_n}-1\right)^{-1},
  \end{equation}
where $c = m_{\Theta}(0) > 0$. If $a_n > b_n$, representation \eqref{mfcn1b} is valid with $c = m_{\Theta}(0) < 0$.
\end{remark}

\begin{proof}[\normalfont \textbf{Proof of Theorem~\ref{uniqres1}}] 
Let us note that
    $$
    \sum_{n \in \mathbb{Z}}\frac{|b_n-a_n|}{1+|a_n|} \leq \sum_{n \in \mathbb{Z}}\frac{|I_n|}{1+\text{dist}(0,I_n)} < \infty ~ \text{and} ~ \sum_{n \in \mathbb{Z}}\frac{|b_n-a_n|}{1+|b_n|} \leq \sum_{n \in \mathbb{Z}}\frac{|I_n|}{1+\text{dist}(0,I_n)} < \infty.
    $$
    Therefore $0 < \prod_{n \in \mathbb{Z}}|a_n|/|b_n| < \infty$ and hence
    \begin{equation}\label{Lcpeqn}
    l := \lim_{y \rightarrow +\infty}m_{\Theta}(iy) = \lim_{y \rightarrow +\infty}c \prod_{n \in \mathbb{Z}}\left(\frac{z}{b_{n}}-1\right)\left(\frac{z}{a_n}-1\right)^{-1} = c \prod_{n \in \mathbb{Z}}\frac{a_n}{b_n},
    \end{equation}
    which implies $L:=\lim_{y \rightarrow +\infty}\Theta(iy) = (l-i)/(l+i) \neq 1$. Therefore $\mu_{\Theta}(\infty) = 0$, i.e.
    $$
    m_{\Theta}(z) = b + iS\mu_{\Theta} = b + \sum_{k \in \mathbb{Z}} \frac{\mu_{\Theta}(a_k)}{\pi}\left(\frac{1}{a_k-z} - \frac{a_k}{1+a_k^2}\right).
    $$
    Now let's observe that $$l = b - \sum_{k \in \mathbb{Z}}\frac{\mu_{\Theta}(a_k)}{\pi}\frac{a_k}{1+a_k^2}.$$ Indeed, we just obtained that the limit of $m_{\Theta}(iy)$ is finite as $y$ goes to $+\infty$. Therefore $m_{\Theta}$ is bounded on $i(\varepsilon,\infty)$ for a fixed $\varepsilon > 0$ and
\begin{align*}
l = \lim_{y \rightarrow +\infty} m_{\Theta}(iy) &= b + \frac{1}{\pi}\sum_{k \in \mathbb{Z}} \lim_{y \rightarrow +\infty} \mu_{\Theta}(a_k)\left(\frac{1}{a_k-iy} - \frac{a_k}{1+a_k^2}\right)\\ &= b - \frac{1}{\pi}\sum_{k \in \mathbb{Z}} \mu_{\Theta}(a_k)\frac{a_k}{1+a_k^2},
\end{align*}
which implies
\begin{equation}\label{repm}
 m_{\Theta}(z) = l + \frac{1}{\pi}\sum_{k \in \mathbb{Z}} \frac{\mu_{\Theta}(a_k)}{(a_k-z)},
\end{equation}
so the only unknown on the right hand side is $l$.

At this point we consider three different cases of our spectral data given in the last item in the theorem statement. 

If $L = \lim_{y \rightarrow +\infty}\Theta(iy)$ is known, then $L$ uniquely determines $l = \lim_{y \rightarrow +\infty}m_{\Theta}(iy)$, since $L = (l-i)/(l+i)$ and $(x-i)/(x+i)$ is an injective function. Therefore by \eqref{repm} $m_{\Theta}$ is uniquely determined. 

If $c = m_{\Theta}(0)$ is known, then using \eqref{repm} we get
$$
l = c - \frac{1}{\pi}\sum_{k \in \mathbb{Z}} \frac{\mu_{\Theta}(a_k)}{a_k},
$$
so $l$ is known and hence $m_{\Theta}$ is uniquely determined.

If $p=\prod_{n \in \mathbb{Z}}a_n/b_n \neq 1$ is known, in order to show uniqueness of $l$ let us consider another MIF $\widetilde{\Theta}$ satisfying the following properties:
\begin{itemize}
    \item $\{\widetilde{\Theta}=1\} = \{a_n\}_{n \in \mathbb{Z}}$,
    \item $\widetilde{\Theta}'(a_n)=\Theta'(a_n)$ for any $n \in \mathbb{Z}$,
    \item $\displaystyle
    \sum_{n \in \mathbb{Z}}\frac{|\widetilde{I}_n|}{1+\text{dist}(0,\widetilde{I}_n)} < \infty
    $ (and hence $\mu_{\widetilde{\Theta}}(\infty) = 0$),
\end{itemize}
where $\widetilde{I}_n = (a_n,\widetilde{b}_n)$ and $\{\widetilde{b}_n\}_{n \in \mathbb{Z}} := \sigma(-\Theta)$. In other words MIFs $\Theta$ and $\widetilde{\Theta}$ share the same Clark measure, i.e. $\mu_{\Theta} = \mu_{\widetilde{\Theta}}$, and $\{\widetilde{I}_n\}_{n}$ satisfy the summability condition, so 
\begin{equation*}
    m_{\widetilde{\Theta}}(z) = i\frac{1+\widetilde{\Theta}(z)}{1-\widetilde{\Theta}(z)} = \widetilde{c} \prod_{n\in \mathbb{Z}}\left(\frac{z}{\widetilde{b}_{n}}-1\right)\left(\frac{z}{a_n}-1\right)^{-1} = \widetilde{l} + \frac{1}{\pi}\sum_{n \in \mathbb{Z}} \frac{\mu_{\Theta}(a_n)}{(a_n-z)}.
\end{equation*}
Let us recall from \eqref{Lcpeqn} that $l = cp$, so we have $l/c = \widetilde{l}/\widetilde{c} = p$. Also by assumption $p \neq 1$, we get $l \neq c$ and $\widetilde{l} \neq \widetilde{c}$. Letting $z=0$ in \eqref{repm}, we also get $c - \widetilde{c} = l - \widetilde{l}$. If we call this value $x$, then we have $$
\frac{l}{c} = \frac{l-x}{c-x}
$$
and hence $x(l-c) = 0$, so $x=0$, i.e $l=\widetilde{l}$ and $m_{\Theta}$ is uniquely determined. 

In all three cases $m_{\Theta}$ is uniquely determined from the given spectral data. Then $\Theta$ is uniquely determined through the one-to-one correspondence between MHFs and MIFs, so we get the desired result.
\end{proof}
%\begin{remark}
%Since $\mu_{\Theta}$ is a discrete measure supported on $\{\Theta = 1\}$ and point masses are given by $\mu_{\Theta} = 2/|\Theta'(a_n)|$, one can read the results so far as if the given conditions are satisfied, then the spectral data consisting of $\{a_n\}_{n \in \mathbb{N}}$ and $\{\Theta^{'}(a_n)\}_{n \in \mathbb{N}}$ uniquely determine $\Theta$. 
%\end{remark}

\begin{proof}[\normalfont \textbf{Proof of Theorem~\ref{uniqres2}}]
Recall that knowing $\mu_{\Theta}$ means knowing $\{a_n\}_{n}$ and $\{\Theta'(a_n)\}_{n}$. We also know that $\mu_{\Theta}(\infty) = 0$, so the MHF corresponding to $\Theta$ has the representation 
    \begin{equation}\label{mfunctionrep2}
    m_{\Theta}(z) = b + iS\mu_{\Theta} = b + \frac{1}{\pi}\sum_{n \in \mathbb{Z}} \mu_{\Theta}(a_n)\left(\frac{1}{a_n-z} - \frac{a_n}{1+a_n^2}\right).
    \end{equation}
The condition $ \int 1/(1+|x|) d\mu_{\Theta}(x) < \infty$ means $\sum_{n \in \mathbb{Z}}\mu(a_n)/(1+|a_n|) < \infty$, which implies that $ \sum_{n \in \mathbb{Z}} a_n\mu(a_n)/(1+a_n^2)
$
is convergent and 
$
\sum_{n \in \mathbb{Z}}\mu(a_n)/(a_n - z)
$
is uniformly bounded by $\sum_{n \in \mathbb{Z}}\mu(a_n)/(1+|a_n|)$ on the set $i(1,\infty)$. These observations together with the representation \eqref{mfunctionrep2} and Lemma \ref{lmm1} imply
$$
c\prod_{n \in \mathbb{Z}}\frac{a_n}{b_n} = \lim_{y \rightarrow \infty} m_{\Theta}(iy) =  b - \frac{1}{\pi}\sum_{n \in \mathbb{Z}} \mu_{\Theta}(a_n)\frac{a_n}{1+a_n^2},
$$
i.e. $\prod_{n \in \mathbb{Z}}a_n/b_n$ is convergent. Therefore 
$$
m_{\Theta}(z) = l + \frac{1}{\pi}\sum_{n \in \mathbb{Z}} \frac{\mu_{\Theta}(a_n)}{a_n-z},
$$
so we get the representation \eqref{repm} again. In order to show uniqueness of $m_{\Theta}$, we can follow the arguments (starting after \eqref{repm}) we used in the proof of Theorem \ref{uniqres1} in all three cases of given spectral data. Uniqueness of $m_{\Theta}$ gives uniqueness of $\Theta$ through the one-to-one correspondence between MHFs and MIFs, so we get the desired result.
\end{proof}

\begin{proof}[\normalfont \textbf{Proof of Theorem~\ref{uniqres3}}]
Let $d > \max\{|a_1|,|b_1|\}$. Since the MHF $m_{\Theta}$ corresponding to $\Theta$ satisfies the infinite product representation \eqref{mfcn1b}, by substituting $z-d$ for $z$ we can keep the derivative values of $\Theta$ same and make $\{a_n+d\}_{n \in \mathbb{N}}$ and $\{b_n+d\}_{n \in \mathbb{N}}$ subsets of $\mathbb{R}_+$. Therefore without loss of generality we assume $\{a_n\}_{n \in \mathbb{N}} \bigcup \{b_n\}_{n \in \mathbb{N}} \subset \mathbb{R}_+$.

Now we are ready to prove part $(1)$. We know that
\begin{equation}\label{Schwarzintrep}
    m_{\Theta}(z) = az + b + iS\mu_{\Theta} = az + b + \frac{1}{\pi}\sum_{n \in \mathbb{N}} \mu_{\Theta}(a_n)\left(\frac{1}{a_n-z} - \frac{a_n}{1+a_n^2}\right), 
\end{equation}
where $a \geq 0$, $b \in \mathbb{R}$ and $\mu_{\Theta}(a_n) \geq 0$ for any $n \in \mathbb{N}$. First let's show that $a=0$. Note that in part $(1)$ we assume $a_n < b_n$, so $m_{\Theta}$ satisfies the infinite product representation \eqref{mfcn1b} with positive $c = m_{\Theta}(0)$. Then partial products of $m_{\Theta}$ are represented as
$$
m_{\Theta}^{(N)}(z) := c\prod_{n=1}^{N}\frac{a_n}{a_n-z}\frac{b_n-z}{b_n} = \sum_{n=1}^N \frac{\alpha_{n,N}}{a_n-z} + c\prod_{n=1}^{N}\frac{a_n}{b_n}, 
$$
where $c > 0$ and $\alpha_{n,N} > 0$ for any $n \in \{1,\cdots,N\}$, $N \in \mathbb{N}$. Let us also note that $\alpha_{n,N}$ converges to $\mu_{\Theta}(a_n)/\pi$ for any fixed $n$ and $\{\mu_{\Theta}(a_n)/a_n^2\}_{n \in \mathbb{N}} \in l^1(\mathbb{N})$. Then for $K < N$, let us consider the difference 
$$
m_{\Theta}^{(N)}(z) - \sum_{n=1}^K \alpha_{n,N}\left(\frac{1}{a_n-z} - \frac{a_n}{1+a_n^2}\right) = \sum_{n=K+1}^{N}\frac{\alpha_{n,N}}{a_n-z} + c\prod_{n=1}^{N}\frac{a_n}{b_n} + \sum_{n=1}^K \alpha_{n,N}\frac{a_n}{1+a_n^2}.
$$
Note that for any $z < 0$, $K < N$ and $N \in \mathbb{N}$, the right-hand side is positive since $\alpha_{n,N}$, $c$, $a_n$, $b_n$ are positive numbers, so the left-hand side is positive for any $z < 0$, $K < N$ and $N \in \mathbb{N}$. First letting $N$ tend to infinity we get
$$
m_{\Theta}(z) - \frac{1}{\pi}\sum_{n=1}^K \mu_{\Theta}(a_n)\left(\frac{1}{a_n-z} - \frac{a_n}{1+a_n^2}\right),
$$
and then letting $K$ tend to infinity we get 
$$
m_{\Theta}(z) - \frac{1}{\pi}\sum_{n\in \mathbb{N}} \mu_{\Theta}(a_n)\left(\frac{1}{a_n-z} - \frac{a_n}{1+a_n^2}\right).
$$
We observed that this difference should be non-negative for $z<0$. However by \eqref{Schwarzintrep} it is nothing but $az + b$ with $a\geq0$, which is possible only if $a=0$, i.e.
\begin{equation}\label{Schwarzintrep2}
    m_{\Theta}(z) =  b + \frac{1}{\pi}\sum_{n \in \mathbb{N}} \mu(a_n)\left(\frac{1}{a_n-z} - \frac{a_n}{1+a_n^2}\right), 
\end{equation}
Therefore 
$$
b = c - \frac{1}{\pi}\sum_{n \in \mathbb{N}} \mu(a_n)\left( \frac{1}{a_n+a_n^3}\right),
$$
so $b$ and hence $m_{\Theta}$ is uniquely determined. Using the one-to-one correspondence between MHFs and MIFs we get the desired result of part $(1)$.

For part $(2)$, let's observe that $-\Theta$ is a MIF and its MHF is $-1/m_{\Theta}$. Using $-1/m_{\Theta}$ as our MHF allows us to swap the roles of $\{a_n\}$ and $\{b_n\}$. Also note that $a_n > b_n$ and $m_{\Theta} < 0$ in part $(2)$. Therefore $-1/m_{\Theta}$ (or $-\Theta$) with the spectral data of part $(2)$ falls to the setting of part $(1)$, so we follow the proof of part $(1)$ and obtain uniqueness of $-1/m_{\Theta}$. Using the one-to-one correspondence between MHFs and MIFs we get the desired result of part $(2)$.
\end{proof}

\begin{remark}\label{rem1}
In the proof of Theorem \ref{uniqres3}, we showed that there is no point mass at infinity. However it does not necessarily imply that $L:=\lim_{y\rightarrow+\infty}\Theta(iy) \neq 1$ (or equivalently $l:=\lim_{y\rightarrow+\infty} m_{\Theta}(iy) < \infty$). If we know that this limit condition is satisfied, then we can replace $c$ in the spectral data with $L$ by following the arguments we used in the proof of Theorem \ref{uniqres1}. Same applies to $p:=\prod_{n \in \mathbb{Z}}a_n/b_n$, if it is convergent and not equal to $1$.
\end{remark}

\begin{proof}[\normalfont \textbf{Proof of Theorem~\ref{uniqres4}}]
From Lemma \ref{lmm1}, without loss of generality we can assume that $m_{\Theta}$ has the representation
\begin{equation*}
m_{\Theta}(z) = c \prod_{n\in \mathbb{N}}\left(\frac{z}{b_{n}}-1\right)\left(\frac{z}{a_n}-1\right)^{-1}.
\end{equation*}
Note that for any $k\in A$, we know
\begin{equation}\label{con1}
\frac{-\mu_{\Theta}(a_k)}{\pi} = \mathrm{Res}(m_{\Theta},a_k) = c(b_k - a_k)\frac{a_k}{b_k} \prod_{n\in \mathbb{N},n\neq k}\left(\frac{a_k}{b_{n}}-1\right)\left(\frac{a_k}{a_n}-1\right)^{-1}.
\end{equation}
Let $m_{\Theta}(z) = f(z)g(z)$, where $f$ and $g$ are two infinite products defined as
\begin{equation*}
 f(z) := c \prod_{n\in A}\left(\frac{z}{b_{n}}-1\right)\left(\frac{z}{a_n}-1\right)^{-1}, \qquad
 g(z) := \prod_{n\in\mathbb{N} \text{\textbackslash} A}\left(\frac{z}{b_{n}}-1\right)\left(\frac{z}{a_n}-1\right)^{-1}
\end{equation*}
Let us observe that at any point of $\{a_n\}_{n \in A}$, the infinite product
\begin{equation}\label{con2}
 g(z)=\prod_{n\in\mathbb{N} \text{\textbackslash} A}\left(\frac{z}{b_{n}}-1\right)\left(\frac{z}{a_n}-1\right)^{-1}
\end{equation}
is also known. Then conditions \eqref{con1} and \eqref{con2} imply that for any $k\in A$, we know
\begin{equation*}
\mathrm{Res}(f,a_k) = \frac{\mathrm{Res}(m_{\Theta},a_k)}{g(a_k)}.
\end{equation*}
Since real zeros and poles of $f$ are simple and interlacing, $f$ is a MHF. If $\Phi$ is the MIF corresponding to $f$, then $\{\Phi = 1\} = \{a_n\}_{n \in A}$, $\{\Phi = -1\} = \{b_n\}_{n \in A}$ and our spectral data become $\{a_n\}_{n \in A}$, $\{\Phi'(a_n)\}_{n \in A}$ and $c = f(0)$, $\lim_{y \rightarrow +\infty}\Phi(iy)$ or $\prod_{n \in A}a_n/b_n$. The convergence condition \eqref{short1seq} implies $0 < \prod_{n \in \mathbb{N}}|a_n|/|b_n| < \infty$ and hence $0 < \prod_{n \in A}|a_n|/|b_n| < \infty$. Therefore applying the proof of Theorem \ref{uniqres1}, we determine $\Phi$ and hence $\{b_n\}_{n \in A}$ uniquely. This means unique recovery of $m_{\Theta}$ and then unique recovery of $\Theta$ through the one-to-one correspondence between MHFs and MIFs.
\end{proof}

\begin{proof}[\normalfont \textbf{Proof of Theorem~\ref{uniqres5}}]
For part $(1)$, following arguments from the proof of Theorem \ref{uniqres4}, we convert our problem to the problem of unique determination of the MIF $\Phi$ from the spectral data $\{a_n\}_{n \in A}$ and $\{\Phi'(a_n)\}_{n \in A}$, where $\{\Phi = 1\} = \{a_n\}_{n \in A}$, $\{\Phi = -1\} = \{b_n\}_{n \in A}$. Since $\{a_n\}_{n \in A}$ is bounded below, by Theorem \ref{uniqres3} part $(1)$ these spectral data uniquely determine $\Phi$ and hence $\{b_n\}_{n \in A}$. This means unique recovery of $m_{\Theta}$ and then unique recovery of $\Theta$ through the one-to-one correspondence between MHFs and MIFs.

For part $(2)$, considering the MHF $-1/m_{\Theta}$ and following arguments from the proof of Theorem \ref{uniqres4}, we convert our problem to the problem of unique determination of the MIF $\Phi$ from the spectral data $\{b_n\}_{n \in A}$ and $\{\Phi'(b_n)\}_{n \in A}$, where $\{\Phi = 1\} = \{b_n\}_{n \in A}$, $\{\Phi = -1\} = \{a_n\}_{n \in A}$. Since $\{a_n\}_{n \in A}$ is bounded below, by Theorem \ref{uniqres3} part $(2)$ these spectral data uniquely determine $\Phi$ and hence $\{a_n\}_{n \in A}$. This means unique recovery of $-1/m_{\Theta}$ and then unique recovery of $\Theta$ through the one-to-one correspondence between MHFs and MIFs.
\end{proof}

\begin{remark}
Remark \ref{rem1} is also valid for Theorem \ref{uniqres5}.    
\end{remark}

\section{Applications to inverse spectral theory of canonical systems}\label{Sec3}

In this section we prove our inverse spectral theorems on canonical Hamiltonian systems. Recall that we introduced these differential systems in \eqref{cHs} and discussed some of their properties in the limit circle case (i.e. $\int_0^L \text{Tr} H(x) dx < \infty$) for $0 < L < \infty$ in Section \ref{Sec1}. Since we are in the limit circle case, we can normalize our Hamiltonian systems by letting the trace to be identically one. Also note that throughout this section $L$ will be arbitrary but fixed.

In order to obtain our results we need another definition. The \textit{transfer matrix} $T(z)$ is the $2\times 2$ matrix solution of \eqref{cHs} with the initial condition identity matrix at $x=0$. Some properties of transfer matrices are given in Theorem 1.2 in \cite{REM18}. Following definition 4.3 in \cite{REM18}, we call collection of all matrices with these properties $TM$, namely the set of matrix functions $T:\mathbb{C} \rightarrow \mathrm{SL}(2,\mathbb{C})$ such that $T$ is entire, $T(0) = I_2$, $\overline{T(\overline{z})} = T(z)$ and if $\mathrm{Im}z \geq 0$, then $i(T^{*}(z)JT(z) - J) \geq 0$. We denote any $T \in TM$ by
$$
T(z) := {\begin{pmatrix}
			A(z) & B(z) \\
			C(z) & D(z) 
		\end{pmatrix}}.
$$
Note that the transfer matrix of any trace normed canonical system on $[0,L]$ satisfies $C'(0) - B'(0) = L$ and $C'(0) - B'(0) \geq 0$ for any $T \in TM$ (see page 106 in \cite{REM18} for explanations). Therefore if we define the disjoint subset
\begin{equation}\label{TML}
    TM(L) := \{ T \in TM~|~ C'(0) - B'(0) = L\},
\end{equation}
then $TM = \cup_{L \geq 0} TM(L)$. The following result shows that $TM$ characterizes all transfer matrices on finite intervals.

\begin{theorem}\label{Remlingthm}\emph{\textbf{(}\cite{REM18}, \textit{Theorem 5.2}\textbf{)}}
Let $L \geq 0$. For every $T \in TM(L)$, there is a unique trace normed canonical Hamiltonian system $H$ on $[0, L]$ such that $T$ is the transfer matrix of $H$. 
\end{theorem}

Next, we focus on connections between $m$-functions and transfer matrices. The entries of $T$ appears in $m$-functions with Dirichlet-Neumann and Dirichlet-Dirichlet boundary conditions, namely $-B/A = m_{0,\pi/2}$ and $-D/C = m_{0,0}$ (page 86, \cite{REM18}). This allows us to obtain unique recovery of transfer matrices from $m$-functions.

\begin{proposition}\label{csprop}
Let $L > 0$. Then the Weyl m-function $m_{0,\pi/2}$ uniquely determines the transfer matrix $T \in TM(L)$.
\end{proposition}
\begin{proof}
Let $T,\widetilde{T} \in TM(L)$ share the same  $m_{0,\pi/2}$, i.e. $-B(z)/A(z) = -\widetilde{B}(z)/\widetilde{A}(z)$. By Theorem 4.22 in \cite{REM18}
\begin{equation}\label{TMeq}
    {\begin{pmatrix}
			\widetilde{A}(z) & \widetilde{B}(z) \\
			\widetilde{C}(z) & \widetilde{D}(z) 
		\end{pmatrix}} = {\begin{pmatrix}
			1 & 0 \\
			az & 1 
		\end{pmatrix}}{\begin{pmatrix}
			A(z) & B(z) \\
			C(z) & D(z) 
		\end{pmatrix}} = {\begin{pmatrix}
			A(z) & B(z) \\
			azA+C(z) & azB+D(z) 
		\end{pmatrix}}
\end{equation}
for some $a \in \mathbb{R}$. Since $T,\widetilde{T} \in TM(L)$, we also know that $C'(0) - B'(0) = \widetilde{C}'(0) - \widetilde{B}'(0) = L$ and $T(0) = \widetilde{T}(0) = I_2$. Therefore by \eqref{TMeq}, $L = \widetilde{C}'(0) - \widetilde{B}'(0) = aA(0) + C'(0) - B'(0) = a + L$ and hence $T = \widetilde{T}$.
\end{proof}
In order to consider general boundary conditions we introduce another notation. Again following \cite{REM18} let
$$
R_{\alpha} := {\begin{pmatrix}
			\cos\alpha& -\sin\alpha \\
			\sin\alpha & \cos\alpha 
		\end{pmatrix}}.
$$
Note that $R_{\alpha}$ is a unitary matrix, $R^{-1}_{\alpha} = R_{-\alpha}$ and $\mathrm{det}R_{\alpha} = 1$. If $f$ is a single variable, then by $R_{\alpha}f$ we mean division of the first entry of the $2\times 1$ vector $R_{\alpha}(f,1)^{\mathrm{T}}$ by its second entry. For example $m_{\alpha,\beta} = R_{-\alpha}m_{0,\beta}$. We will use the same notation for the transfer matrices. 

Now we are ready to prove our inverse spectral results. Let's start with the proof of Theorem \ref{invspecprb2}, since the proof of Theorem \ref{invspecprb1} will require handling both spectra in the same MHF and hence introducing generalized $m$-functions and $R$-matrices.
\begin{proof}[\normalfont \textbf{Proof of Theorem~\ref{invspecprb2}}]
In order to use Theorem \ref{uniqres3}, first let's show that $m_{\alpha,\beta}(0)$ solely depends on $\alpha$ and $\beta$. We discussed the identity $m_{0,\beta} = R_{\alpha}m_{\alpha,\beta}$. Also recalling the identity $m_{0,\beta}(z) = T^{-1}(z)\cot\beta$ ($(3.4)$ in \cite{REM18}), we get
$$
m_{\alpha,\beta}(0) = R_{-\alpha}m_{0,\beta}(0) = R_{-\alpha}\cot\beta = \frac{\cos\alpha\cos\beta + \sin\alpha\sin\beta}{-\sin\alpha\cos\beta + \cos\alpha\sin\beta} = \cot(\beta-\alpha).
$$
Therefore by Theorem \ref{uniqres3} part $(1)$, the spectral measure $\mu_{\alpha,\beta}$ and boundary conditions $\alpha$ and $\beta$ uniquely determine the Weyl inner function $\Theta_{\alpha,\beta}$ and hence the Weyl m-function $m_{\alpha,\beta}$ since there is a one-to-one correspondence between MIFs and MHFs. By the identity $m_{0,\beta} = R_{\alpha}m_{\alpha,\beta}$, the $m$-function $m_{\alpha,\beta}$ and the boundary condition $\alpha$ uniquely determine $m_{0,\beta}$. We still need to pass to $\pi/2$ from general $\beta$ in order to use Proposition \ref{csprop}. For this, we use a transformation of $H$, namely $H_{\gamma} := R_{\gamma}^{\mathrm{T}}HR_{\gamma}$. If $m^{(\gamma)}$ denotes the $m$-function of $H_{\gamma}$, then $m^{(\gamma)}_{0,\beta-\gamma}(z) = R_{-\gamma}m_{0,\beta}(z)$ (see Theorem 3.20 and explanation below that in \cite{REM18}). By letting $\gamma = \beta - \pi/2$, we can uniquely determine $m^{(\gamma)}_{0,\pi/2}$ from the knowledge of $m_{0,\beta}$ and $\beta$. Now by Proposition \ref{csprop} we obtain the transfer matrix of $H_{\beta - \pi/2}$ and then by Theorem \ref{Remlingthm} the Hamiltonian $H_{\beta - \pi/2}$ uniquely. Finally, recalling $H = R_{\gamma}H_{\gamma}R_{\gamma}^{\mathrm{T}}$, we get unique determination of $H$ from uniqueness of $H_{\beta-\pi/2}$ and $\beta$. 
\end{proof}

\begin{proof}[\normalfont \textbf{Proof of Theorem~\ref{invspecprb1}}]
In order to handle both spectra in the same MHF let's introduce generalized $m$-functions and $R$-matrices: $$m_{\alpha_1,\alpha_2,\beta}(z) := \frac{-\sin(\alpha_2)f_1(0) + \cos(\alpha_2)f_2(0)}{-\sin(\alpha_1)f_1(0) + \cos(\alpha_1)f_2(0)}, \qquad R_{\alpha_1,\alpha_2} := {\begin{pmatrix}
			-\sin\alpha_2& \cos\alpha_2 \\
			-\sin\alpha_1 & \cos\alpha_1 
		\end{pmatrix}}.$$
Note that $\det R_{\alpha_1,\alpha_2} = \sin(\alpha_1 - \alpha_2)$, so $R_{\alpha_1,\alpha_2}$ is invertible. Also
$m_{\alpha_1,\alpha_2,\beta} = R_{\alpha_1,\alpha_2} m_{0,\beta}$ and hence $m_{\alpha_1,\alpha_2,\beta} = R_{\alpha_1,\alpha_2}R^{-1}_{\alpha_1}m_{\alpha_1,\beta}$. Moreover $\sigma_{\alpha_1,\beta}$ and $\sigma_{\alpha_2,\beta}$ are sets of poles and zeros of $m_{\alpha_1,\alpha_2,\beta}$ respectively. Another critical observation is that $m_{\alpha_1,\alpha_2,\beta}$ is a MHF since $m_{0,\beta}$ is a MHF. Therefore we can introduce corresponding MIF $\Theta_{\alpha_1,\alpha_2,\beta}$ and spectral measure $\mu_{\alpha_1,\alpha_2,\beta} = \sum\gamma_{\alpha_1,\alpha_2,\beta}^{(n)}\delta_{a_n}$. Keeping this notations in mind, we need to pass from $\mu_{\alpha_1,\beta}$ to $\mu_{\alpha_1,\alpha_2,\beta}$. We can do this using two observations: firstly both measures are supported on $\sigma_{\alpha_1,\beta}$, secondly the point masses or norming constants are related by the identity $\gamma^{(n)}_{\alpha_1,\alpha_2,\beta} = \sin(\alpha_1 - \alpha_2)\gamma^{(n)}_{\alpha_1,\beta}$, which is also valid for the point masses at infinity. First observation follows from the fact that the $m_{\alpha_1,\alpha_2,\beta}$ and $m_{\alpha_1,\beta}$ share the same set of poles $\sigma_{\alpha_1,\beta}$. Let's prove the second observation. For simplicity we use the following notation: $s_k := \sin(\alpha_k)$ and $c_k := \cos(\alpha_k)$. We know that $a_n$ is a pole for both $m$-functions, so
$$
\mathrm{Res}(m_{\alpha_1,\beta},z=a_n) = \big(c_1f_1(0) + s_1f_2(0)\big)\Big|_{z=a_n}\lim_{z \rightarrow a_n}\frac{z-a_n}{-s_1f_1(0) + c_1f_2(0)}
$$
and similarly
$$
\mathrm{Res}(m_{\alpha_1,\alpha_2,\beta},z=a_n) = \big(-s_2f_1(0) + c_2f_2(0)\big)\Big|_{z=a_n}\lim_{z \rightarrow a_n}\frac{z-a_n}{-s_1f_1(0) + c_1f_2(0)}.
$$
Therefore
$$
\frac{\mathrm{Res}(m_{\alpha_1,\beta},z=a_n)}{\mathrm{Res}(m_{\alpha_1,\alpha_2,\beta},z=a_n)} = \frac{c_1\frac{f_1(0)}{f_2(0)}\Big|_{z=a_n} + s_1}{-s_2\frac{f_1(0)}{f_2(0)}\Big|_{z=a_n} + c_2}.$$
Since $a_n$ is a pole, at $z=a_n$, $-s_1f_1(0)+c_1f_2(0) = 0$, i.e. $(f_1(0)/f_2(0))|_{z=a_n} = c_1/s_1$. Hence 
$$
\frac{\mathrm{Res}(m_{\alpha_1,\beta},z=a_n)}{\mathrm{Res}(m_{\alpha_1,\alpha_2,\beta},z=a_n)} = \frac{\frac{c_1^2}{s_1}+s_1}{-s_2\frac{c_1}{s_1}+c_2} = \frac{c_1^2+s_1^2}{-s_2c_1+s_1c_2} = \frac{1}{\sin(\alpha_1-\alpha_2)}
$$
if $s_1 \neq 0$ and $c_1 \neq 0$. One can check other cases similarly and get 
$$
\frac{\mathrm{Res}(m_{\alpha_1,\beta},z=a_n)}{\mathrm{Res}(m_{\alpha_1,\alpha_2,\beta},z=a_n)} = -\frac{c_1}{s_2} \qquad \text{and} \qquad \frac{\mathrm{Res}(m_{\alpha_1,\beta},z=a_n)}{\mathrm{Res}(m_{\alpha_1,\alpha_2,\beta},z=a_n)} = -\frac{s_1}{c_2}
$$
for the cases $s_1 = 0$ and $c_1 = 0$ respectively. In all three cases $\mathrm{Res}(m_{\alpha_1,\beta},z=a_n)$ is given in terms of $\mathrm{Res}(m_{\alpha_1,\alpha_2,\beta},z=a_n)$, $\alpha_1$ and $\alpha_2$ for any $n$. Finally recalling that the residue of a MHF at a pole is $-1/\pi$ times the corresponding point mass by Herglotz representation \eqref{Schwarzintrep}, we get unique determination of $\mu_{\alpha_1,\alpha_2,\beta}$ from $\mu_{\alpha_1,\beta}$, $\alpha_1$ and $\alpha_2$. The point mass at infinity can be handled similarly by comparing residues of $m_{\alpha_1,\alpha_2,\beta}(1/z)$ and $m_{\alpha_1,\beta}(1/z)$ at 0. 

Now, by the identity $m_{\alpha_1,\alpha_2,\beta}(0) = R_{\alpha_1,\alpha_2}\cot(\beta)$ and Theorem \ref{uniqres1}, the spectral measure $\mu_{\alpha_1,\alpha_2,\beta}$ and boundary conditions $\alpha_1$, $\alpha_2$ and $\beta$ uniquely determine the inner function $\Theta_{\alpha_1,\alpha_2,\beta}$ and hence the m-function $m_{\alpha_1,\alpha_2,\beta}$. We know that $m_{\alpha_1,\beta} = R_{\alpha_1}R^{-1}_{\alpha_1,\alpha_2}m_{\alpha_1,\alpha_2,\beta}$, so $m_{\alpha_1,\beta}$ is uniquely determined. Then we can follow the same steps we used in the proof of Theorem \ref{invspecprb2}, starting at the unique determination of $m_{\alpha,\beta}$ step, and get the desired result.
\end{proof}

\begin{proof}[\normalfont \textbf{Proof of Theorem~\ref{invspecprb3}}]
By Theorem \ref{uniqres4} and the arguments we used at the beginning of the proof of Theorem \ref{invspecprb1}, the given spectral data uniquely determine the MIF $\Theta_{\alpha_1,\alpha_2,\beta}$ and hence the MHF $m_{\alpha_1,\alpha_2,\beta}$. Then we can follow the same steps we used in the proof of Theorem \ref{invspecprb1} and get the desired result.
\end{proof}

\begin{proof}[\normalfont \textbf{Proof of Theorem~\ref{invspecprb4}}]
By Theorem \ref{uniqres5} and the arguments we used at the beginning of the proof of Theorem \ref{invspecprb2}, the given spectral data uniquely determine the MIF $\Theta_{\alpha_1,\alpha_2,\beta}$ and hence the MHF $m_{\alpha_1,\alpha_2,\beta}$. Then we can follow the same steps we used in the proof of Theorem \ref{invspecprb1} and get the desired result.
\end{proof}

\section{Acknowledgments}
Part of this work was conducted at Georgia Institute of Technology, where the author was a postdoc of Svetlana Jitomirskaya. The author thanks funding from NSF DMS-2052899, DMS-2155211, and Simons 681675.

\bibliographystyle{abbrv}
\bibliography{references}

\end{document}